\documentclass[a4paper,11pt]{article}
    \usepackage{amsmath}
    \usepackage{amssymb}
    \usepackage{latexsym}
    \usepackage{paralist}
    \usepackage{xypic}

\newtheorem{Theorem}{Theorem}[section]
\newtheorem{Lemma}[Theorem]{Lemma}
\newtheorem{Corollary}[Theorem]{Corollary}
\newtheorem{Proposition}[Theorem]{Proposition}
\newtheorem{Remark}[Theorem]{Remark}

\newtheorem{Definition}[Theorem]{Definition}

\begin{document}
\title{A classification of Newton polygons of $L$-functions on
polynomials\footnote{This research was supported by the National Center for Mathematics and Interdisciplinary Sciences, CAS}}
\author{Fusheng Leng\footnote{communication author: leng@amss.ac.cn}, Banghe Li\\                     
        { Key Laboratory of Mathematics Mechanization, } \\
         { Academy of Mathematics and Systems Sciences, CAS}}
\maketitle                           

\begin{abstract}
\noindent 
Considering the $L$-function of exponential sums associated to a
polynomial over a finite field $\mathbb{F}_{q}$, Deligne proved that a reciprocal root's $p$-adic order is a rational number in the interval $[0, 1]$. Based on hypergeometric theory, in this paper we improve this result that there are only finitely many possible forms of Newton polygons for the
$L$-function of degree $d$ polynomials independent of p, when $p$ is larger than a constant $D^{*}$(Theorem \ref{theorem 4.3}), i.e., a reciprocal root's $p$-adic order has form $\frac{up-v}{D^{*}(p-1)}$ in which $u, v$ have finitely many possible values. Furthermore, when $p>D^{*}$, to determine the Newton polygon is only to determine it for any two specified primes $p_{1}, p_{2}>D^{*}$ in the same residue class of $D^{*}$(Theorem \ref{theorem 4.5}).
\end{abstract}

\section{Introduction}
\indent   
Let $\mathbb{F}_{q}$ ($q=p^{m}$) be the finite field of $q$ elements
with characteristic $p$ and $\mathbb{F}_{q^{r}}$ be the extension of
$\mathbb{F}_{q}$ of degree $r$. Let $\zeta_{p}$ be a fixed primitive
$p$-th root of unity in the complex numbers. For any Laurent
polynomial $f(x_{1},\dots,x_{n})\in
\mathbb{F}_{q}[x_{1},x_{1}^{-1},\dots,x_{n},x_{n}^{-1}]$, we form
the exponential sum
$$
S_{r}(f)=\sum_{x_{i}\in\mathbb{F}_{q^{r}}}\zeta_{p}^{\textrm{Tr}_{\mathbb{F}_{q^{r}}/\mathbb{F}_{q}}(f(x_{1},\cdots,x_{n}))}.
$$
The $L$-function is defined by
$$
L(f,T) = \textrm{exp}(\sum_{r=1}^{\infty}S_{r}(f)\frac{T^{r}}{r}) .
$$

Consider the case $n=1$, and $f$ is a polynomial with degree $d<p$.
It follows from Weil's work on the Riemann hypothesis for
function fields with characteristic $p$ that this $L$-function is
actually a polynomial of degree $d-1$. We can write it as

$$
L(f,T) = (1-\omega_{1}T)\cdots (1-\omega_{d-1}T).
$$

Another work on it of Weil is that the reciprocal roots $\omega_{1},
\cdots \omega_{d-1}$ are \textit{q-Weil numbers of weight 1}, i.e.,
algebraic integers, all of whose conjugates have complex absolute
$q^{\frac{1}{2}}$. Moreover, for any prime $l\neq p$, they are
$l$-adic units, that is, $|\omega_{i}|_{l}=1$, while
$|\omega_{i}|_{p}=q^{r_{i}}$ with $r_{i}\in \mathbb{Q} \cap [0,1]$. Deligne proved in general that $|\omega_{i}|=q^{\frac{u_{i}}{2}}$ with $u_{i}\in \mathbb{Z}\cap [0, 2n]$, and $|\omega_{i}|_{p}=q^{-r_{i}}$ with $r_{i}\in \mathbb{Q}\cap [0, n]$ (\cite{3} and \cite{10}).

A natural question is how to determine their $q$-adic absolute value, or
equivalently to determine $r_{i}$. In other words, one would like to
determine the Newton polygon $NP_{q}(f)$ of $L(f,T)$ where $NP_{q}$
means the Newton polygon taken with respect to the valuation
$\mathbf{ord}_{q}$ normalized by $\mathbf{ord}_{q}q=1$(cf. \cite{4},
Chapter IV for the link between the Newton polygon of a polynomial
and the valuations of its roots). There is an elegant general answer
to this problem when $p \equiv 1 \ \mathbf{mod} \ d$, $p\geq 5$:
the Newton polygon $NP_{q}(f)$ has vertices(cf. \cite{5} Theorem
7.5)

$$
(n, \frac{n(n+1)}{2d})_{1\leq n\leq d-1}
$$
This polygon is often called the \textit{Hodge polygon} for
polynomials of degree $d$, and denoted by $HP(d)$. We know that
$NP_{q}(f)$ lies above $HP(d)$.

Unfortunately, for general prime $p$, there is no such exact answer
of $NP_{q}(f)$. $NP_{q}(f)$ has been only determined for polynomials with degree 3, 4,
6(\cite{2}, \cite{6} and \cite{7} respectively), in which the case of degree 6 has some essential mistakes pointed by us in \cite{13}, via comparing with our results. In \cite{13}, we gave another method to calculate $NP_{q}(f)$ for polynomials with degree 5, 6, and the first $[\sqrt{2d}]+2$ slopes in general for $q=p$.

Another result on $NP_{q}(f)$ is Hui Zhu's result(\cite{8} and \cite{9}), concerning the one-dimensional case of Wan's conjecture \cite{10} as follows. Let the polygon
$NP_{q}(f)$ be the Newton polygon obtained from the reduction of $f$
modulo a prime above $p$ in the field defined by the coefficients of
$f$. Then there is a Zariski dense open subset $\mathcal{U}$ defined
over $\mathbb{Q}$ in the space of polynomial of degree $d$ such
that, for every $f$ in $\mathcal{U}(\bar{\mathbb{Q}})$, we have
$\lim_{p\rightarrow \infty} NP_{q}(f)=HP(d)$.

The Grothendieck specialization theorem implies
that, in the case of dimension one, for $p>d$, the lowest
Newton polygon

$$
GNP(d,p)=\inf NP_{q}(f)
$$
exists when $f$ runs over polynomials of degree $d$ over $\mathbb{F}_{q}$, which is called \textit{generic Newton polygon}.

Blache and F\'{e}rard determined explicitly both the generic polygon $GNP(d,p)$
and the associated \textit{Hasse polynomial} $H_{d,p}$ for $p\geq 3d$ \cite{1}.

Because of Deligne's work, when expanding $L(f,T)=\det(\mathbf{I}_{d-1}-T\Gamma^{\tau^{m-1}}\cdots \Gamma))$(cf.\cite{1}, Proposition 1.1) directly, we can only consider the items whose $p$-adic orders are smaller than a fixed number. A fact that this partial algebraic sum can be expressed as a finite linear combination of hypergeometric polynomials will be proved in section 3 ( Lemma \ref{lemma 3.2} ). Combining with H.S.Wilf and D.Zeilberger's work \cite{12}, we improved Deligne's work that there are only finitely many possible forms of Newton polygons independent of $p$, when $p$ is larger than a constant $D^{*}$ ( Theorem \ref{theorem 4.3} ).

Furthermore, to
determine $NP_{q}(f)$ for $p>D^{*}$, we need only to determine it for any two fixed primes
$p_{1}, p_{2}>D^{*}$ in the same residue class of $p$ modular $D^{*}$ ( Theorem \ref{theorem 4.5} ).

We use $p$-adic cohomology of Dwork, Robba and
others, especially Blache and F\'{e}rard's detailed description in their
paper \cite{1}.

The rest of the paper is organized as below: In Section \ref{section 2.1} we introduce some preliminaries on hypergeometric and holonomic functions, which are needed for proving our theorems in Section \ref{section 4}. We also introduce some concepts and results in Section \ref{section 2.2} as basis of our work. In Section \ref{section 3} we expand $L(f,T)$ in detail. Based on Lemma \ref{lemma 3.2}, we will prove our main theorem in Section \ref{section 4}.
\begin{center}
\end{center}
\textbf{Acknowledgements} We are indebted to Professor Daqing Wan for introducing us into this field, and for his continuous guide and encouragement to our works.

\section{Preliminary}

\subsection{Hypergeometric and holonomic functions}\label{section 2.1}

In this section, we introduce the so called Fundamental Corollary in theory of hypergeometric and holonomic functions needed for us \cite{12}.

\begin{Definition}
A function $F(a_{1}, \cdots , a_{m})$ of $m$ discrete variables is a hypergeometric term if for every $a_{i}$,

$$
\frac{E_{a_{i}}F}{F}=\frac{P_{i}}{Q_{i}}, i=1, \cdots , m.
$$
where $P_{i}$, $Q_{i}$ are all polynomials in the variables $(a_{1}, \cdots , a_{m})$, and $E_{a}$ are the shift operators acting on functions $f(a, \textbf{b})$ by changing $a$ to $a+1$, i.e.

$$
E_{a}f(a, \mathbf{b})=f(a+1, \mathbf{b}).
$$

\end{Definition}

Phrased otherwise, $F$ is a solution of the system of linear recurrence equations

$$
(Q_{i}E_{a_{i}}-P_{i})F=0, i=1, \cdots , m.
$$
If the dimention of the space of solutions of that system is $finite$, the functions $F$ are called $holonomic$.

In \cite{12} it is shown how to check for holonomicity, and in particular it is proved that the following class of $proper-hypergeometric functions$ are holonomic (we omit the continuous variables part here).

\begin{Definition}
A term $F(a_{1}, \cdots , a_{m})$ of $m$ discrete variables is proper-hypergeometric, if it has the form

$$
P(a_{1}, \cdots , a_{m})\prod_{i=1}^{I}(e_{0}^{(i)}+\sum_{j=1}^{m}e_{j}^{(i)}a_{j})!^{g_{i}}
$$
where $P(a_{1}, \cdots , a_{m})$ is a polynormial and $e_{j}^{(i)}$ and $g_{i}$ are integers.

\end{Definition}

We need two more concepts for leading to the Fundamental Corollary.

\begin{Definition}
A function $F(a_{1}, \cdots , a_{m})$ vanishes at infinity if for every variable $a_{i}$,

$$
\lim_{|a_{i}|\rightarrow \infty}F(\mathbf{a})=0.
$$

\end{Definition}

\begin{Definition}
An integral-sum

$$
g(\mathbf{n}):=\sum_{\mathbf{k}}F(\mathbf{n},\mathbf{k})
$$
is pointwise trivially evaluable, if for every specific specialization of the auxiliary variables (parameters) $\mathbf{n}$ there is an algorithm that will evaluate it.
\end{Definition}

\begin{Corollary}
(\textbf{Fundamental Corollary})
Let $F(n,\mathbf{k})$ be hypergeometric and holonomic(both hold if it is proper-hypergeometric) in the discrete variables $n$ and $\mathbf{k}$, and vanishes at infinity for every fixed $n$, then

$$
f(n):= \sum_{\mathbf{k}}F(n,\mathbf{k})
$$
satisfies a linear recurrence equation with polynomial coefficients
$$
P(N,n)f(n)=0.
$$
where $N$ is the shift operators acting on functions by changing $n$ to $n+1$.

\end{Corollary}

H.S.Wilf and D.Zeilberger gave a method to find such $P(N,n)$, and gave an effective upper bound for the $N$-degree of $P(N,n)$ \cite{12}. More detailed, assume

$$
F(n,\mathbf{k})=R(n,\mathbf{k})\frac{\prod_{s=1}^{p}(a_{s}n+\mathbf{b}_{s}\cdot \mathbf{k}+c_{s})!}{\prod_{s=1}^{q}(u_{s}n+\mathbf{v}_{s}\cdot \mathbf{k}+w_{s})!}
$$
where $\mathbf{k}=(k_{1}, \cdots , k_{r})$, $R$ is a polynomial, the $a$'s, $u$'s, $\mathbf{b}$'s and $\mathbf{v}$'s are integers that contain no additional parameters, and the $c$'s
and $w$'s are integers that may involve unspecified parameters (in our case the prime $p$ is such parameter in $c$'s and $w$'s ).

Let
$$
f_{n}(\mathbf{x})=\sum_{\mathbf{k}}F(n,\mathbf{k})\mathbf{x^{k}}.
$$
Then there exist polynomials $\alpha_{j}(n,\mathbf{x})$ independent of $\mathbf{k}$ satisfying
$$
\sum_{j=0}^{J}\alpha_{j}(n,\mathbf{x})f_{n-j}(\mathbf{x})=0
$$
where the coefficients
$$
\alpha_{j}(n,\mathbf{x})=\sum_{0\leq i_{1}, \cdots , i_{r} \leq I}\alpha_{\mathbf{i},j}(n)\mathbf{x^{i}}
$$
and $I, J$ are both bounded by $a, \mathbf{b}, u, \mathbf{v}, deg(R)$.

\subsection{$p$-Adic differential operators and exponential sums}\label{section 2.2}

Let $f(x):=\alpha_{d}x^{d}+\cdots +\alpha_{1}x$, $\alpha_{d}\neq 0$,
be a polynomial of degree $d<p$, over the field $\mathbb{F}_{q}$, and let $g(x):=a_{d}x^{d}+\cdots +a_{1}x \in \mathcal{O}_{m}[x]$ be the polynomial whose coefficients are the Teichm$\ddot{u}$ller lifts of those of $f$. Let $A:=B(0, 1^{+})\backslash B(0, 1^{-})$. We consider the space $\mathcal{H}^{\dag}(A)$ of overconvergent analytic functions on $A$. Define the function $H:=\exp(\pi g(x))$ and let $D$ be the differential operator
$$
D:=x\frac{d}{dx}+\pi xg'(x) \ \ \ (= H^{-1}\circ x\frac{d}{dx}\circ H).
$$

Then for every $n\in \mathbb{Z}$, $x^{n}$ can be written uniquely as

$$
x^{n}\equiv \sum_{i=0}^{d-1}a_{ni}x^{i} \ \mathbf{mod} \ D\mathcal{H}^{\dag}(A)
$$
for some $a_{ni}\in \mathcal{K}_{m}(\pi), 0\leq i\leq d-1$, where $\mathcal{K}_{m}$ is an unramified extension of degree $m$ of the $p$-adic numbers field $\mathbb{Q}_{p}$.

We define the power series $\theta(x):=\exp(\pi x-\pi x^{p})$ and $F(x):=\prod_{i=1}^{d}\theta(a_{i}x^{i}):=\sum_{n\geq 0}h_{n}x^{n}$. Define a mapping $\psi_{q}$ on $\mathcal{H}^{\dag}(A)$ by $\psi_{q}f(x):=\frac{1}{q}\sum_{z^{q}=x}f(z)$. Let $\beta$ be the endomorphism of $\mathcal{H}^{\dag}(A)$ defined by $\beta=\psi_{p}\circ F$; then $\tau^{-1}\circ \beta$($\tau$ being the Frobenius) commutes with $D$ up to a factor $p$, and passes to the quotient, giving an endomorphism $\tau^{-1}\circ \beta$ of $W$, the $\mathcal{K}_{m}(\zeta_{p})$-vector space with basis $\mathcal{B}$.

Let $M:=\textrm{Mat}_{\mathcal{B}}(\overline{\beta})$ be the matrix of $\overline{\beta}$ in the basis $\mathcal{B}$, and $m_{ij}$($0\leq i,j\leq d-1$) be the coefficients of $M$. From the description of $F$, we can write $m_{ij}=h_{pi-j}+\sum_{n\geq d}h_{np-j}a_{ni}$(cf. \cite{11}, 7.10). Define $\Gamma :=(m_{ij})_{1\leq i,j\leq d-1}$, then

$$
L(f,T)=\det(\mathbf{I}_{d-1}-T\Gamma^{\tau^{m-1}}\cdots \Gamma)
$$
(cf. \cite{1}, Proposition 1.1).

\section{Finite sum expression of $\mathcal{M}_{n}$}\label{section 3}

Let $f(x):=\alpha_{d}x^{d}+\cdots +\alpha_{1}x$, $\alpha_{d}\neq 0$,
be a polynomial of degree $d<p$, over the field $\mathbb{F}_{q}$. Denote $L(f,T)=1+\sum_{n=1}^{d-1}\mathcal{M}_{n}T^{n}$, our aim is to
determine every $\mathbf{ord}_{q}\mathcal{M}_{n}$.

We expand $\det(\mathbf{I}_{d-1}-T\Gamma^{\tau^{m-1}}\cdots \Gamma))$, i.e.

\begin{equation}\label{equation 1}
\mathcal{M}_{n}=\sum_{1\leq u_{1}<\cdots<u_{n}\leq d-1}\sum_{\sigma\in S_{n}}\textrm{sgn}(\sigma)\prod_{i=1}^{n}(\sum_{1\leq k_{1}, \cdots , k_{m-1}\leq d-1}m_{u_{i},k_{1}}^{\tau^{m-1}}m_{k_{1},k_{2}}^{\tau^{m-2}}\cdots m_{k_{m-1},u_{\sigma(i)}}).
\end{equation}

\subsection{Finite sum expression of $m_{ij}$}

Set $\theta(x):=\Sigma_{i\geq 0}b_{i}x^{i}$; then we have ( Lemma 1.2 in \cite{1} )

$(i)\ \mathbf{ord}_{p}b_{i}\geq \frac{i}{p-1} \ for \ 0\leq i\leq p^{2}-1;$

$(ii)\mathbf{ord}_{p}b_{i}\geq \frac{(p-1)i}{p^{2}} \ for \ i\geq p^{2}.$

Since
$$
h_{n}=\sum_{i_{1}+ \cdots + di_{d}= n}a_{1}^{i_{1}}\cdots a_{d}^{i_{d}}b_{i_{1}}\cdots b_{i_{d}},
$$
then

\begin{Lemma}\label{lemma 3.1}
We have the relations

$(i)\ \mathbf{ord}_{p}h_{i}\geq \frac{i}{d(p-1)} \ for \ 0\leq i\leq p^{2}-1;$

$(ii)\mathbf{ord}_{p}h_{i}\geq \frac{(p-1)i}{dp^{2}} \ for \ i\geq p^{2}.$
\end{Lemma}

\begin{proof}
For $0\leq i\leq p^{2}-1$ and $i_{1}+ \cdots + di_{d}=i$, we have $0\leq i_{1}, \cdots, i_{d}\leq p^{2}-1$ and

$$
\mathbf{ord}_{p}b_{i_{1}}\cdots b_{i_{d}}=\mathbf{ord}_{p}b_{i_{1}}+ \cdots + \mathbf{ord}_{p}b_{i_{d}}\geq \frac{i_{1}}{p-1}+ \cdots + \frac{i_{d}}{p-1}\geq \frac{i}{d(p-1)},
$$
then $(i)$ holds.

For $i\geq p^{2}$ and $i_{1}+ \cdots + di_{d}=i$, since $\frac{1}{p-1}\geq \frac{p-1}{p^{2}}$, we have

$$
\mathbf{ord}_{p}b_{i_{1}}\cdots b_{i_{d}}=\mathbf{ord}_{p}b_{i_{1}}+ \cdots + \mathbf{ord}_{p}b_{i_{d}}\geq \frac{(p-1)i_{1}}{p^{2}}+ \cdots + \frac{(p-1)i_{d}}{p^{2}}\geq \frac{(p-1)i}{dp^{2}},
$$
then $(ii)$ holds.
\end{proof}

Furthermore,

$$
h_{n}=\sum_{k\geq 0}\pi^{k}\sum_{\substack{\sum_{i=1}^{d}i(m_{i}+(p-1)n_{i})=n, \\ \sum_{i=1}^{d}m_{i}=k}}\prod_{i=1}^{d}\frac{(-1)^{n_{i}}}{m_{i}!}\left(
  \begin{array}{cc}
    m_{i}\\
    n_{i}\\
  \end{array}
\right)a_{i}^{m_{i}+(p-1)n_{i}},
$$

following the definition of $m_{ij}$ in Section \ref{section 2.2}, we have

\begin{equation}\label{equation 2}
m_{ij}=\sum_{r>0}h_{rp-j}a_{ri}=\sum_{r>0}a_{ri}\sum_{k\geq 0}\pi^{k}\sum_{\substack{\sum_{l=1}^{d}l(m_{l}+(p-1)n_{l})=rp-j, \\ \sum_{l=1}^{d}m_{l}=k}}\prod_{l=1}^{d}\frac{(-1)^{n_{l}}}{m_{l}!}\left(
  \begin{array}{cc}
    m_{l}\\
    n_{l}\\
  \end{array}
\right)a_{l}^{m_{l}+(p-1)n_{l}}.
\end{equation}

Since the $p$-adic order of the reciprocal roots of $L(f,T)$ are all smaller than 1, we can consider only the items in (\ref{equation 1}) whose $p$-adic orders are smaller than $n$. If $p>3$, then $\frac{(p-1)(rp-j)}{dp^{2}}-\frac{r-i}{d(p-1)}>n$ when $r\geq 3nd$ and $1\leq i, j\leq d-1$. Following Lemma 1.1 in \cite{1}, and equation (\ref{equation 1}) and Lemma \ref{lemma 3.1}, we should consider only the part $r<3nd$ in the sum of equation (\ref{equation 2}).

If $p\geq 3d^{2}$, then any $n_{l}$ in equation (\ref{equation 2}) must smaller than $3nd$ when $r<3nd$. We can rewrite $m_{ij}$ to

$$
m_{ij}=\sum_{r\geq 3nd}h_{rp-j}a_{ri}+
$$
\begin{equation}\label{equation 3}
\sum_{0<r<3nd}a_{ri}\sum_{0\leq n_{1}, \cdots , n_{d}<3nd}\sum_{k\geq 0}\pi^{k}\sum_{\substack{\sum_{l=1}^{d}l(m_{l}+(p-1)n_{l})=rp-j, \\ \sum_{l=1}^{d}m_{l}=k}}\prod_{l=1}^{d}\frac{(-1)^{n_{l}}}{m_{l}!}\left(
  \begin{array}{cc}
    m_{l}\\
    n_{l}\\
  \end{array}
\right)a_{l}^{m_{l}+(p-1)n_{l}}.
\end{equation}

\subsection{Finite sum expression of $\mathcal{M}_{n}$}

Let
$$
s_{l}=k-\sum_{t=1}^{l}m_{t}
$$
for $1\leq l\leq d-1$ and $s_{0}=k$ and $s_{d}=0$. Then
$$
m_{l}=s_{l-1}-s_{l}
$$
for $1\leq l\leq d$.

We transform (\ref{equation 3}) into

$$
m_{ij}=\sum_{r\geq 3nd}h_{rp-j}a_{ri}+
$$
$$
\sum_{0<r<3nd}a_{ri}\sum_{0\leq n_{1}, \cdots , n_{d}<3nd}\sum_{k\geq 0}\pi^{k}\sum_{\sum_{l=1}^{d}l(s_{l-1}-s_{l}+(p-1)n_{l})=rp-j}\prod_{l=1}^{d}\frac{(-1)^{n_{l}}}{(s_{l-1}-s_{l})!}\left(
  \begin{array}{cc}
    s_{l-1}-s_{l}\\
    n_{l}\\
  \end{array}
\right)a_{l}^{s_{l-1}-s_{l}+(p-1)n_{l}}.
$$
Note that the sum $\sum_{l=1}^{d}l(s_{l-1}-s_{l}+(p-1)n_{l})=rp-j$ is equivalent to

$$
s_{d-1}=rp-j-\sum_{l=1}^{d-1}s_{l-1}+\sum_{l=1}^{d}l(p-1)n_{l},
$$
we can omit $\sum_{\sum_{l=1}^{d}l(s_{l-1}-s_{l}+(p-1)n_{l})=rp-j}$ by substituting $s_{d-1}$ into the expression.		

Let
$$
\omega=\{\sigma, u_{i}, k_{j}, r_{i,j}, n_{i,j,l}\}_{i=1,\cdots ,n;j=1,\cdots ,m;l=1,\cdots ,d}
$$
where $0\leq n_{i,j,l}<3nd$.

Let $p_{i,j}=r_{i,j}p-k_{j}$ for $i=1,\cdots ,n;j=1,\cdots ,m-1$ and $p_{i,m}=r_{i,m}p-u_{\sigma(i)}$ for $i=1,\cdots ,n$.

Denote ($\{s_{i,j,0}=k_{i,j}$, $s_{i,j,d-1}=p_{i,j}-\sum_{l=1}^{d-1}s_{l-1}+\sum_{l=1}^{d}l(p-1)n_{i,j,l}$ and $s_{i,j,d}=0$)

$$
F_{\omega}(\{s_{i,j,0},s_{i,j,1},\cdots ,s_{i,j,d-2}\})=
$$
$$
\prod_{i=1,\cdots ,n;j=1,\cdots ,m}\prod_{l=1}^{d}\frac{(-1)^{n_{i,j,l}}}{(s_{i,j,l-1}-s_{i,j,l})!}\left(
  \begin{array}{cc}
    s_{i,j,l-1}-s_{i,j,l}\\
    n_{i,j,l}\\
  \end{array}
\right)a_{l}^{s_{i,j,l-1}-s_{i,j,l}+(p-1)n_{i,j,l}}.
$$

Then $F_{\omega}$ is hypergeometric and holonomic. Replacing $k_{n,m}$ with variable $0\leq a$ and $0<k\leq p-1$ such that $k_{n,m}=a(p-1)+k-\sum_{(i,j)\neq (n,m)}k_{i,j}$.

Furthermore, let

$$
G_{\omega,a}(k)=\sum_{s_{i,j,0},s_{i,j,1},\cdots ,s_{i,j,d-1}}F_{\omega}(\{s_{i,j,0},s_{i,j,1},\cdots ,s_{i,j,d-2}\})
$$

Then we have

$$
\mathcal{M}_{n}=\sum_{1\leq u_{1}<\cdots<u_{n}\leq d-1}\sum_{\sigma\in S_{n}}\textrm{sgn}(\sigma)\sum_{1\leq k_{1}, \cdots , k_{m-1}\leq d-1}\sum_{0<r_{i,j}<3nd}
$$
\begin{equation}
\prod_{1\leq i\leq n}a_{r_{i,1},u_{i}}^{\tau^{m-1}}\prod_{1\leq i\leq n,1<j\leq m}a_{r_{i,j},k_{j-1}}^{\tau^{m-j}}(\sum_{0\leq n_{i,j,1},\cdots , n_{i,j,d}<3nd}\sum_{0\leq a}(-p)^{a}\sum_{0<k\leq p-1}\pi^{k}G_{\omega,a}(k))+\mathcal{N}_{n}
\end{equation}
where the $p$-adic order of $\mathcal{N}_{n}$ is not smaller than $d$.

Since $k_{i,j}\leq r_{i,j}p-1<3ndp$, we have
$$
a(p-1)+k=\sum_{i,j}k_{i,j}<3dmn^{2}p,
$$
thus $a\leq 3dmn^{2}$. Therefore we can consider only

$$
\mathcal{\widetilde{M}}_{n}=\sum_{1\leq u_{1}<\cdots<u_{n}\leq d-1}\sum_{\sigma\in S_{n}}\textrm{sgn}(\sigma)\sum_{1\leq k_{1}, \cdots , k_{m-1}\leq d-1}\sum_{0<r_{i,j}<3nd}
$$
\begin{equation}
\prod_{1\leq i\leq n}a_{r_{i,j},u_{i}}^{\tau^{m-1}}\prod_{1\leq i\leq n,1<j\leq m}a_{r_{i,j},k_{j-1}}^{\tau^{m-j}}(\sum_{0\leq n_{i,j,1},\cdots , n_{i,j,d}<3nd}\sum_{0\leq a\leq 3dmn^{2}}(-p)^{a}\sum_{0<k\leq p-1}\pi^{k}G_{\omega,a}(k))
\end{equation}
instead of $\mathcal{M}_{n}$. We have

\begin{Lemma}\label{lemma 3.2}
$\mathbf{ord}_{p}\mathcal{M}_{n}$ is congruent with $\frac{c}{p-1}$ or some $\frac{c}{p-1}+\mathbf{ord}_{p}\sum_{0<k\leq p-1}\pi^{k}G_{\omega,a}(k)$ $(\mathbf{mod} \ 1)$, where $-3mn^{2}<c\leq 0$ is an integer independent of $p$.
\end{Lemma}

\begin{proof}
By Lemma 1.1. in \cite{1}, $0\geq \mathbf{ord}_{p}a_{r,i}:=\frac{s}{p-1}\geq -\frac{r-i}{d(p-1)}$. When we restrict $r$ to $r<3nd$, we have $0\geq s> -3n$.

For any $r_{i,j}$, $u_{i}$ and $k_{j}$, let $\mathbf{ord}_{p}\prod_{1\leq i\leq n}a_{r_{i,j},u_{i}}^{\tau^{m-1}}\prod_{1\leq i\leq n,1<j\leq m}a_{r_{i,j},k_{j-1}}^{\tau^{m-j}}:=\frac{z}{p-1}$. Then it is within the range
$$
0\geq \frac{z}{p-1}> -\frac{3n^{2}m}{p-1}.
$$

Besides,
$$
\sum_{1\leq u_{1}<\cdots<u_{n}\leq d-1}\sum_{\sigma\in S_{n}}\textrm{sgn}(\sigma)\sum_{1\leq k_{1}, \cdots , k_{m-1}\leq d-1}
$$
\begin{equation}\label{equation 6}
\sum_{0<r_{i,j}< 3nd}\prod_{1\leq i\leq n}a_{r_{i,j},u_{i}}^{\tau^{m-1}}\prod_{1\leq i\leq n,1<j\leq m}a_{r_{i,j},k_{j-1}}^{\tau^{m-j}}(\sum_{0\leq n_{i,j,1},\cdots , n_{i,j,d}<3nd}\sum_{0\leq a\leq 3dmn^{2}}(-p)^{a})
\end{equation}
is a finite sum independent of $p$. In other words, $\mathbf{ord}_{p}\mathcal{M}_{n}$ is congruent with the $p$-adic order of one of the finitely many numbers items in the sum of (\ref{equation 6}) $(\mathbf{mod} \ 1)$.

The result follows.
\end{proof}

\section{Finite possible forms of Newton polygon}\label{section 4}

Since $F_{\omega}$ is hypergeometric and holonomic, following the Fundamental Corollary in holonomic theory, there exist a non-zero operator $P_{\omega,a}(K,k)$  ( a polynomial of $k$ and shift operator $K$ defined by $K\circ f(k)=f(k+1)$ ) annihilate $G_{\omega,a}(k)$, i.e.
$$
P_{\omega,a}(K,k)G_{\omega,a}(k)\equiv 0.
$$

By the definition of $G_{\omega,a}(k)$, for any prime $p$ we have
$$
\mathbf{ord}_{p}G_{\omega,a}(k)\equiv 0 (\mathbf{mod} \ 1)
$$

Suppose $P_{\omega,a}(K,k)=\sum_{i=0}^{e-1}K^{i}P_{i}(k)\in \mathbb{Z}[K,k]$, where $e$ is independent of $p$ ( see the end of Section \ref{section 2.1} ). Let $P_{i_{0}}, \cdots , P_{i_{h}}$ be all of such $P_{i}$ satisfying

$$
p \nmid P_{i_{j}}
$$
where $j=0, \cdots , r$ and $e-1\geq i_{0}> \cdots > i_{r}\geq 0$. ( If each of $P_{i}$ is divisible by $p$, we can divide by $p$ to reduce $P_{\omega,a}(K,k)$ until such $P_{i_{j}}$ appears. )

\begin{Proposition}\label{proposition 4.1}
If

$$
\mathbf{ord}_{p}\sum_{0<k\leq p-1}\pi^{k}G_{\omega,a}(k)\equiv \frac{k_{\omega,a}}{p-1} (\mathbf{mod} \ 1)
$$
where $e<k_{\omega,a}\leq p-e$, then

$$
P_{i_{0}}(k_{\omega,a})\equiv 0 (\mathbf{mod} \ p).
$$

\end{Proposition}

\begin{proof}

Considering the equation
$$
P_{\omega,a}(K,k_{\omega,a}-i_{0})G_{\omega,a}(k_{\omega,a}-i_{0})=0
$$
i.e.,
\begin{equation}\label{equation 7}
\sum_{i=0}^{e-1}K^{i}P_{i}(k_{\omega,a}-i_{0})G_{\omega,a}(k_{\omega,a}-i_{0})=\sum_{i=0}^{e-1}P_{i}(k_{\omega,a}-i_{0}+i)G_{\omega,a}(k_{\omega,a}-i_{0}+i)=0.
\end{equation}
We will show that

\begin{equation}\label{equation 8}
\mathbf{ord}_{p}G_{\omega,a}(k_{\omega,a})< \mathbf{ord}_{p}P_{i}(k_{\omega,a}-i_{0}+i)G_{\omega,a}(k_{\omega,a}-i_{0}+i)
\end{equation}
for all $0\leq i\leq e-1$.

Since

$$
\mathbf{ord}_{p}G_{\omega,a}(k)\equiv 0 (\mathbf{mod} \ 1),
$$
the congruences of $\mathbf{ord}_{p}\pi^{k}G_{\omega,a}(k) \ (\mathbf{mod} \ 1)$ for $k=1, \cdots , p-1$ are all different.

Furthermore, by the definition,
$$
\mathbf{ord}_{p}\sum_{0<k\leq p-1}\pi^{k}G_{\omega,a}(k)\equiv \frac{k_{\omega,a}}{p-1} (\mathbf{mod} \ 1),
$$
then for any $k\neq k_{\omega,a}$, $\mathbf{ord}_{p}\pi^{k_{\omega,a}}G_{\omega,a}(k_{\omega,a})< \mathbf{ord}_{p}\pi^{k}G_{\omega,a}(k)$.

For any $0<k<k_{\omega,a}$, from the discussion above, we have

$$
\frac{k_{\omega,a}}{p-1}+\mathbf{ord}_{p}G_{\omega,a}(k_{\omega,a})=\mathbf{ord}_{p}\pi^{k_{\omega,a}}G_{\omega,a}(k_{\omega,a})< \mathbf{ord}_{p}\pi^{k}G_{\omega,a}(k)=\frac{k}{p-1}+\mathbf{ord}_{p}G_{\omega,a}(k)
$$
$$
<\frac{k_{\omega,a}}{p-1}+\mathbf{ord}_{p}G_{\omega,a}(k),
$$
thus
$$
\mathbf{ord}_{p}G_{\omega,a}(k_{\omega,a})< \mathbf{ord}_{p}G_{\omega,a}(k).
$$

For any $k_{\omega,a}<k\leq p-1$, we also have

$$
\mathbf{ord}_{p}G_{\omega,a}(k_{\omega,a})< \frac{k_{\omega,a}}{p-1}+\mathbf{ord}_{p}G_{\omega,a}(k_{\omega,a})=\mathbf{ord}_{p}\pi^{k_{\omega,a}}G_{\omega,a}(k_{\omega,a})< \mathbf{ord}_{p}\pi^{k}G_{\omega,a}(k)=\frac{k}{p-1}+\mathbf{ord}_{p}G_{\omega,a}(k)
$$
$$
\leq 1+\mathbf{ord}_{p}G_{\omega,a}(k),
$$
i.e.,
$$
\mathbf{ord}_{p}G_{\omega,a}(k_{\omega,a})<1+\mathbf{ord}_{p}G_{\omega,a}(k).
$$

Recall the definition of $i_{0}$,
$$
p \mid P_{i}
$$
for all $i>i_{0}$, then (\ref{equation 8}) holds. Combining with (\ref{equation 7}), we have

$$
\mathbf{ord_{p}}P_{i_{0}}(k_{\omega,a})>0.
$$
It shows that either $P_{i_{0}}(k_{\omega,a})=0$ or $\mathbf{ord}_{p}P_{i_{0}}(k_{\omega,a})\geq 1$, i.e.,

$$
P_{i_{0}}(k_{\omega,a})\equiv 0 (\mathbf{mod} \ p).
$$

\end{proof}

Since prime $p$ is arbitrary, following Proposition \ref{proposition 4.1} we have

\begin{Corollary}\label{corollary 4.2}
If $e<k_{\omega,a}\leq p-e$, then $k_{\omega,a}|_{p=0}$ is a rational root of $P_{i_{0}}|_{p=0}$, i.e.,

$$
P_{i_{0}}|_{p=0}(k_{\omega,a}|_{p=0})=0
$$

\end{Corollary}

If $0<k_{\omega,a}\leq e$ or $p-e<k_{\omega,a}\leq p-1$ for some fixed index $\omega$, by the definition of $k_{\omega,a}$, i.e., $\mathbf{ord}_{p}\sum_{0<k\leq p-1}\pi^{k}G_{\omega,a}(k)\equiv \frac{k_{\omega,a}}{p-1} \ (\mathbf{mod} \ 1)$, the congruence
$$
\frac{c}{p-1}+\mathbf{ord}_{p}\sum_{0<k\leq p-1}\pi^{k}G_{\omega,a}(k) \ (\mathbf{mod} \ 1)
$$
in Lemma \ref{lemma 3.2} with range $(\frac{-3mn^{2}-e+1}{p-1}, \frac{e+1}{p-1})$ is independent of $p$ in the finite sum (\ref{equation 6}). Combining with Corollary \ref{corollary 4.2}, we have

\begin{Theorem}\label{theorem 4.3}
Let $D^{*}$ be the least common multiple of all denominators of rational roots of such $P_{i_{0}}|_{p=0}$. Suppose $p> D^{*}$ and $p$ in a fixed residue class of $D^{*}$. Denote $p^{*}$ the inverse of $p$ $(\mathbf{mod} \ D^{*})$. Thus
$$
lim_{p\rightarrow \infty}\mathbf{ord}_{p}\mathcal{M}_{n}=\frac{-p^{*}r}{D^{*}} \ (\mathbf{mod} \ 1)
$$
where $r=0$ or $\frac{r}{D^{*}}$ is congruent with a rational root of some $P_{i_{0}}|_{p=0}$ $(\mathbf{mod} \ 1)$
\end{Theorem}

\begin{proof}
Note that the sum of (\ref{equation 6}) has finitely many numbers of items, by Lemma \ref{lemma 3.2}, $\mathbf{ord}_{p}\mathcal{M}_{n}$ must be congruent with one of
$$
\frac{k_{w,a}+c}{p-1} \ (\mathbf{mod} \ 1)
$$
in which $c$ is an integer with $-3mn^{2}<c\leq 0$.

If $k_{\omega,a}\leq e$ or $k_{\omega,a}> p-e$, when $p\rightarrow \infty$ the limit of $\frac{k_{w,a}+c}{p-1} \ (\mathbf{mod} \ 1)$ is equal to 0.

Following Corollary \ref{corollary 4.2}, if $e<k_{\omega,a}\leq p-e$, then $P_{i_{0}}|_{p=0}(k_{\omega,a}|_{p=0})=0$, i.e. $k_{\omega,a}|_{p=0}$ is a rational root of $P_{i_{0}}|_{p=0}$.

Suppose $k_{\omega,a}|_{p=0}=\frac{r}{D^{*}}$, then $k_{\omega,a}$ has the form
$$
k_{\omega,a}=\frac{sp+r}{D^{*}}.
$$
Note that $p$ is prime to $D^{*}$ and $k_{\omega,a}$ is an integer, i.e., $sp+r \equiv 0 \ (\mathbf{mod} \ D^{*})$, then $s\equiv -p^{*}r \ (\mathbf{mod} \ D^{*})$.

So $lim_{p\rightarrow \infty}\frac{k_{\omega,a}}{p-1}=\frac{s}{D^{*}}=\frac{-p^{*}r}{D^{*}} \ (\mathbf{mod} \ 1)$ is determined only by $\frac{r}{D^{*}}$, the rational root of $P_{i_{0}}|_{p=0}$. In this case
$$
lim_{p\rightarrow \infty}\mathbf{ord}_{p}\mathcal{M}_{n}=\frac{-p^{*}r}{D^{*}} \ (\mathbf{mod} \ 1).
$$
The theorem is proved.
\end{proof}

\begin{Remark}
Combining with Lemma \ref{lemma 3.2}, this theorem shows that the Newton polygon of $L(f, T)$ has finitely many possible forms when $p>D^{*}$. This means, we can classify the Newton polygons according to their limits when $p\rightarrow \infty$. In detail, the finite sum of (\ref{equation 6}) has
$$
t=\left(
  \begin{array}{cc}
    d-1\\
    n\\
  \end{array}
\right)n!(d-1)^{m-1}(3nd-1)^{n(m-1)}(3nd)^{dn(m-1)}(3dmn^{2}+1)
$$
items, and then $\mathbf{ord}_{p}\mathcal{M}_{n}$ has at most $t$ numbers of possible values when $p\rightarrow \infty$.
\end{Remark}

Suppose $p> D^{*}$ in a fixed residue class of $D^{*}$ and $\mathbf{ord}_{p}\mathcal{M}_{n}=\frac{up-v}{D^{*}(p-1)}$, i.e.,
\begin{equation}\label{equation 9}
\mathbf{ord}_{p}\mathcal{M}_{n}=\frac{u}{D^{*}}\cdot \frac{p}{p-1}-\frac{v}{D^{*}}\cdot \frac{1}{p-1}.
\end{equation}
If we know the values of $\mathbf{ord}_{p}\mathcal{M}_{n}$ on two specified primes $p_{1},p_{2}>D^{*}$ which are in the same residue class of $D^{*}$ with $p$, saying $\mathbf{ord}_{p_{1}}\mathcal{M}_{n}=r_{1}$, $\mathbf{ord}_{p_{2}}\mathcal{M}_{n}=r_{2}$, then

$$
\frac{u}{D^{*}}=\frac{r_{1}(p_{1}-1)-r_{2}(p_{2}-1)}{p_{1}-p_{2}}
$$
is determined. Thus
$$
\frac{v}{D^{*}}=(\frac{u}{D^{*}}\frac{p_{1}}{p_{1}-1}-r_{1})(p_{1}-1)
$$
is also determined. So $\mathbf{ord}_{p}\mathcal{M}_{n}$ is determined in general by (\ref{equation 9}). i.e., we have

\begin{Theorem}\label{theorem 4.5}
Let $p>D^{*}$. To determine the Newton polygon of $L(f,T)$, we need only calculate it on two specified values of prime $p$ in each residue class of $D^{*}$.
\end{Theorem}

\end{document}